\documentclass[12pt,a4paper]{article}
\usepackage{amsmath}
\usepackage{amsfonts}
\usepackage{latexsym,amsthm,amscd}
\usepackage[left=2cm,bottom=2.4cm,top=2.4cm,right=2cm]{geometry}
\usepackage[colorlinks]{hyperref}
\usepackage[all]{xy}
\usepackage{graphicx}
\allowdisplaybreaks
\newcounter{alphthm}
\newtheorem{prop}{Proposition}[section]
\newtheorem{thm}{Theorem}[section]
\newtheorem{lem}[thm]{Lemma}
\newtheorem{cor}{Corollary}[section]

\theoremstyle{definition}

\newtheorem{ex}{Example}[section]
\newcommand{\be}{\begin{equation}}
\newcommand{\ee}{\end{equation}}

\title{On General Spherically Symmetric Finsler Metrics}
\author{T. Khani-Moghaddam$^{1}$, M. Rafie-Rad$^{1}$\footnote{Corresponding Author}, A. Tayebi$^{2}$
\\\footnotesize{$^{1}$ \textit{Department of Mathematics, University of  Mazandaran, Babolsar, Iran}}
\\\footnotesize{$^{2}$ \textit{Department of Mathematics, University of Qom, Qom, Iran}}
\\\footnotesize\textit{t.khanimoghaddam@stu.umz.ac.ir; rafie-rad@umz.ac.ir; akbar.tayebi@gmail.com}}
\numberwithin{equation}{section}
\begin{document}
\maketitle
\begin{abstract}
In this paper, we study some non-Riemannian curvature properties of general spherically symmetric Finsler metrics. First, we prove that every general spherically symmetric Finsler metric is semi {\it C}-reducible. Then, we find the necessary and sufficient condition under which a general spherically symmetric Finsler metric has vanishing mean stretch curvature.
\end{abstract}
%---------------------------------------------------------------------------------------------------------------------
\section{Introduction}
In 2003, Shen investigated one of the main problems in Finsler geometry, namely, classification of projectively flat Finsler metrics with constant flag curvature \cite{Sh}.  By this investigations, he success to express projectively flat Finsler metrics of constant flag curvature by Taylor expansions or some interesting algebraic formulas. One of the discovered metric can not be expressed in the usual known Finsler metrics. Indeed the obtained metric was not belong to the class $(\alpha,\beta)$-metrics. This Finsler metric defined as follows:
\begin{align}
F(x,y)=\bigg\lbrace1+\langle a,x\rangle +\frac{(1-|x|^2)\langle a,y\rangle}{A+\langle x,y\rangle}\bigg\rbrace\frac{(A+\langle x,y\rangle)^2}{(1-|x|^2)^2A},\label{SM}
\end{align}
where $x\in\mathbb{R}^n$, $y\in T_{x}\mathbb{R}^n$,  $|,|$ and $\langle ,\rangle$ are the standard Euclidean norm and inner product on $\mathbb{R}^n$, respectively. $|a|<1$ is a real constant and $A:=\sqrt{|y|^2-(|x|^2|y|^2-\langle x,y\rangle^2)}$. The metric \eqref{SM} is an interesting projectively flat Finsler metric with flag curvature ${\bf K}=0$. By putting $a=0$, one can get the square projective Finsler metric constructed by Berwald in \cite{B}.

Let us define $r:=|y|, u:=|x|^2, s:=\langle x,y\rangle/|y|, v=\langle a,x\rangle$ and $t:=\langle a,y\rangle/|y|$. Then, one can rewrite \eqref{SM} as follows
\begin{align}\label{eq2}
F=r\Big\lbrace1+v+\frac{(1-u)t}{\sqrt{1-u+s^2}+s}\Big\rbrace\frac{(\sqrt{1-u+s^2}+s)^2}{(1-u)^2\sqrt{1-u+s^2}}.
\end{align}
The metric $F$ in \eqref{eq2} is also given in the form $F=r\phi(u,s,v,t)$, where $a=a_{i}y^i$ is a constant 1-form  and  $\phi$ is a $C^\infty$ function. In a special case $a=0$, the it reduces to a spherically symmetric Finsler metric $F(x,y)=r\phi(u,s)$, where $(x,y)\in T\mathbb{B}^n(r_{0})\backslash\{0\}$. Considering this fact, the  Finsler metrics  $F=r\phi(u,s,v,t)$ are called general spherically symmetric Finsler metrics \cite{LL}. It is remarkable that, if $a\neq 0$, then the Finsler metric  $F=r\phi(u,s,v,t)$ is neither spherically symmetric nor general $(\alpha,\beta)$-metric \cite{LL}. This means that the class of general spherically symmetric Finsler metrics is a rich and important category of Finsler metrics.

The Cartan torsion of Finsler metrics measure their distance from Riemannian metrics. Also, the special faces of Catran torsion distinguish some Finsler metrics. In \cite{MS}, Matsumoto-H\={o}j\={o} proved that a Finsler metric on a manifold of dimension $n\geq3$ is a Randers metric or Kropina metric if and only if it is C-reducible. In \cite{MSh}, Matsumoto-Shibata showed that every $(\alpha,\beta)$-metric is semi-C-reducible. In \cite{TB}, Tayebi-Barati proved that every spherically symmetric Finsler metric $F=u\phi(r,s)$ on a domain $\mathcal{U}\subseteq\mathbb{R}^n$ is semi C-reducible. In the paper, we generalize all of mentioned results and prove the following.
\begin{thm}\label{THM1}
Every general spherically symmetric Finsler metric $F=r\phi(u,s,v,t)$ on a domain $\mathcal{U}\subseteq\mathbb{R}^n$ is semi $C$-reducible.
\end{thm}

\bigskip

The notion of weakly  stretch metrics introduced by Najafi-Tayebi in \cite{NT} as an extension of stretch metrics. Every Finsler metric with vanishing mean stretch curvature is called weakly stretch metric. In \cite{CL}, Chen-Liu studied weakly stretch Randers metrics and proved that any weakly stretch Randers metrics must be of isotropic S-curvature. In \cite{TBS}, Tayebi-Bahadori-Sadeghi find a equation that characterizes weakly stretch spherically symmetric Finsler metrics. It is interesting to study general spherically symmetric with vanishing mean stretch curvature. In this paper, we prove the following.
\begin{thm}\label{THM2}
Let $F=r\phi(u,s,v,t)$ be a general spherically symmetric Finsler metric on a domain $\mathcal{U}\subseteq\mathbb{R}^n$. Then $F$ is a weakly stretch metric if and only if the following PDE holds:
\begin{align}\label{02}
rT(x^ia^j-x^ja^i)+(sT+Z)(a^iy^j-a^jy^i)+(W-tT)(x^iy^j-x^jy^i)=0,
\end{align}
where $T$, $Z$ and $W$ defined in section $4$.
\end{thm}
%------------------------------------------------------------------------------------------------------------------
\section{Preliminary}
Let $M$ be a $C^{\infty}$-manifold of dimension $n$, and  $TM=\bigcup_{x \in M}T_{x}M$ denotes its tangent bundle. Let us define $TM_{0}:=TM\backslash\{0\}$ as the slit tangent bundle of $M$. Then, a Finsler metric on $M$ is a function $ F:TM\rightarrow [0,\infty )$ such that satisfies the following properties:\\
(i) $F$ is $C^\infty$ on $TM_0$;\\
(ii) $F$ is positively 1-homogeneous on the fibers of tangent  bundle $TM$, i.e., 
\[
F(x,\lambda y)=\lambda F(x,y), \ \ \  \forall \lambda>0.
\]
(iii) The quadratic form $\textbf{g}_y:T_xM \times T_xM\rightarrow \mathbb{R}$ is positive-definite on $T_xM$
\[
\textbf{g}_{y}(u,v):={1 \over 2}\frac{\partial^2}{\partial s\partial t} \Big[  F^2 (y+su+tv)\Big]_{s=t=0},\quad
u,v\in T_xM.
\]
The inner product $\textbf{g}_{y}$ is called the fundamental form in the direction $y$. The pair  $(M,F)$ is called a  Finsler manifold.

\bigskip

Let $x\in M$ and $F_x:=F|_{T_xM}$. To measure the non-Euclidean feature of $F_x$, one can define $\mathbf{C_y}: T_xM\times T_xM \times T_xM\rightarrow \mathbb{R}$ by
\begin{align}\label{0C}
\mathbf{C}_y(u,v,w):=\frac{1}{2}\frac{d}{dt}\Big[\mathbf{g}_{y+tw}(u,v)\Big]_{t=0},\quad u,v,w\in T_xM.
\end{align}
The family $\mathbf{C}:=\lbrace\mathbf{C}_y\rbrace_{y\in TM_0}$ is called the Cartan torsion. It is well known that $\mathbf{C}=0$ if and only if $F$ is Riemannian.

Also, for $y\in T_xM_0$, one can define $\mathbf{I}_y:T_xM\rightarrow\mathbb{R}$ by
\[
\mathbf{I}_y(u):=\sum_{i=1}^ng^{ij}(y)\mathbf{C}_y(u,\partial_i,\partial_j),
\]
where $\lbrace\partial_i\rbrace$ is a basis for $T_xM$ at $x\in M$. The family $\mathbf{I}:=\lbrace\mathbf{I}_y\rbrace_{y\in TM_0}$ is called the mean Cartan torsion. By definition, $\mathbf{I}_y(u):=I_i(y)u^i$, where $I_i=g^{jk}C_{ijk}$.

\bigskip

A non-Riemannian Finsler metric $F=F(x,y)$ on a manifold $M$ of dimension $n\geq 3$ is called semi-C-reducible if its Cartan tensor is given by following
\begin{align}\label{SC}
\mathbf{C}_y(u,v,w)=\frac{\mathcal{P}}{n+1}\Big\lbrace \mathbf{I}_y(u)\mathbf{h}_y(v,w)+\mathbf{I}_y(v)\mathbf{h}_y(u,w)+\mathbf{I}_y(w)\mathbf{h}_y(u,v)\Big\rbrace +\frac{\mathcal{Q}}{{\Vert \mathbf{I}\Vert}^2}\mathbf{I}_y(u)\mathbf{I}_y(v)\mathbf{I}_y(w),
\end{align}
where $\mathcal{P}=\mathcal{P}(x,y)$ and $\mathcal{Q}=\mathcal{Q}(x,y)$ are scalar functions on $TM$ satisfying $\mathcal{P}+\mathcal{Q}=1$ and ${\Vert \mathbf{I}\Vert}^2=I^mI_m$.  In local coordinates, we have
\begin{align}
C_{ijk}=\frac{\mathcal{P}}{n+1}\big\lbrace h_{ij}I_k+ h_{jk}I_i+ h_{ik}I_j\big\rbrace +\frac{\mathcal{Q}}{\Vert\mathbf{I}\Vert^2}I_iI_jI_k,
\end{align}
If $\mathcal{Q}=0$, then $F$ reduces to a C-reducible metric \cite{TPN}\cite{TS}. We remark that in \cite{MM0} Matsumoto was proved that every $(\alpha,\beta)$-metric on a manifold $M$ of dimension $n\geq3$ is semi C-reducible.

\bigskip

For $y\in T_xM$, define $\mathbf{L}_y:T_xM\times T_xM\times T_xM\rightarrow\mathbb{R}$ by $\mathbf{L}_y(u,v,w):=L_{ijk}u^iv^jw^k$, where
\begin{align}
L_{ijk}:=-\frac{1}{2}FF_{y^l}\frac{\partial^3G^l}{\partial y^i\partial y^j\partial y^k}.\label{L}
\end{align}
The quantity $\mathbf{L}$ is called the Landsberg curvature and $F$ is called a Landsberg metric if $\mathbf{L}=0$.

\bigskip

For $y\in T_xM$, define $\mathbf{J}_y:T_xM\rightarrow\mathbb{R}$ by
\begin{align}\label{Jg}
\mathbf{J}_y(u)=\sum_{i=1}^ng^{ij}(y)\mathbf{L}_y(u,\partial_i,\partial_j),
\end{align}
where $\lbrace\partial_i\rbrace$ is a basis for $T_xM$ at $x\in M$. The family $\mathbf{J}:=\lbrace\mathbf{J}_y\rbrace_{y\in TM_0}$ is called the mean Landsberg curvature. AFinsler metric $F$ is called a weakly Landsberg metric if $\mathbf{J}=0$.

\bigskip

The weakly stretch curvature $\bar{\Sigma}_y: T_xM\times T_xM\rightarrow\mathbb{R}$ is defined by $\bar{\Sigma}_y(u,v):=\bar{\Sigma}_{ij}(y)u^iv^j$, where
\begin{align}\label{barS}
\bar{\Sigma}_{ij}:=2(J_{i|j}-J_{j|i}).
\end{align}
where $``|"$ denotes the horizontal covariant derivation with respect to the Berward connection of $F$. A Finsler metric is said to be weakly stretch metric if $\bar{\Sigma}$=0.

%--------------------------------------------------------------------------------------------------------------------------------------------
\section{Proof of Theorem \ref{THM1}}
%--------------------------------------------------------------------------------------------------------------------------------------------
The fundamental tensor of   general spherically symmetric Finsler metric $F=r\phi(u,s,v,t)$ is given by
\begin{equation}\label{eq8}
\begin{split}
g_{ij}=&c_{0}\delta_{ij}+c_{1}a_{i}a_{j}+c_{2}\frac{y_{i}}{r}\frac{y_{j}}{r}+c_{3}\Big(a_{j}\frac{y_{i}}{r}+a_{i}\frac{y_{j}}{r}\Big)+c_{4}\Big(x_{j}\frac{y_{i}}{r}
+x_{i}\frac{y_{j}}{r}\Big)\\
&+c_{5}(a_{j}x_{i}+a_{i}x_{j})+c_{6}x_{i}x_{j},
\end{split}
\end{equation}
where
\begin{align*}
c_{0}&=\phi^2-s\phi\phi_{s}-t\phi\phi_{t},\\
c_{1}&=\phi_{t}^2+\phi\phi_{tt},\\
c_{2}&=s^2(\phi_{s}^2+\phi\phi_{ss})+t^2(\phi_{t}^2+\phi\phi_{tt})+2ts(\phi_{s}\phi_{t}+\phi\phi_{st})-s\phi\phi_{s}-t\phi\phi_{t},\\
c_{3}&=\phi\phi_{t}-s(\phi_{s}\phi_{t}+\phi\phi_{st}-t(\phi_{t}^2+\phi\phi_{tt}),\\
c_{4}&=\phi\phi_{s}-s(\phi_{s}^2+\phi\phi_{ss}-t(\phi_{s}\phi_{t}+\phi\phi_{st}),\\
c_{5}&=\phi_{s}\phi_{t}+\phi\phi_{st},\\
c_{6}&=\phi_{s}^2+\phi\phi_{ss}.
\end{align*}
For more details, see \cite{LL}.  Then, one can rewrite \eqref{eq8} as follows
\[
g_{ij}=c_{0}\left(F_{ij}+\gamma\frac{y_{i}}{r}\frac{y_{j}}{r}\right),
\]
where
\begin{align*}
&\gamma=-\frac{c_{3}}{c_{0}},\quad F_{ij}=E_{ij}+\theta N_{i}N_{j},\quad\theta=\frac{c_{5}}{c_{0}},\quad N_{i}=a_{i}+x_{i},\\
&E_{ij}=D_{ij}+\xi M_{i}M_{j},\quad\xi=\frac{c_{3}}{c_{0}},\quad M_{i}=a_{i}+\frac{y_{i}}{r},\\
&D_{ij}=B_{ij}+\epsilon a_{i}a_{j},\quad\epsilon=\frac{c_{1}-c_{3}-c_{5}}{c_{0}},\\
&B_{ij}=A_{ij}+\lambda L_{i}L_{j},\quad\lambda=\frac{c_{2}}{c_{0}},\quad L_{i}=\frac{c_{4}}{c_{2}}x_{i}+\frac{y_{i}}{r},\\
&A_{ij}=\delta_{ij}+\mu x_{i}x_{j},\quad\mu=\frac{c_{6}c_{2}-c_{4}^2-c_{5}c_{2}}{c_{0}c_{2}}.
\end{align*}
The inverse of the metric tensor is given by
\begin{align}\label{eq91}
g^{ij}=c_{0}^{-1}\Big\lbrace\delta^{ij}-\zeta x^ix^j-\tau L^iL^j-\nu(B^{ij})^2a_{i}a_{j}-\sigma M^iM^j-\kappa N^iN^j-\alpha(F^{ij})^2y_{i}y_{j}\Big\rbrace ,
\end{align}
where
\begin{eqnarray*}
\!\!&&\!\!\zeta =\frac{\mu}{1+\mu u},\\
\!\!&&\!\!\tau =\frac{\lambda}{1+\lambda L^2},\\
\!\!&&\!\! L^i=\omega x^i+\frac{y^i}{r},\\
\!\!&&\!\!\nu =\frac{\epsilon}{1+\epsilon a^2},\\\\
\!\!&&\!\! B^{ij}a_{j}=b_{1}x^i+b_{2}\frac{y^i}{r}+a^i,\\
\!\!&&\!\!\sigma =\frac{\xi}{1+\xi M^2},\\
\!\!&&\!\! M^i= d_{1}x^i+d_{2}\frac{y^i}{r}+d_{3}a^i,\\
\!\!&&\!\!\kappa =\frac{\theta}{1+\theta N^2},\\
\!\!&&\!\! N^i= e_{1}x^i+e_{2}\frac{y^i}{r}+e_{3}a^i,\\
\!\!&&\!\!\alpha =\frac{\gamma}{1+\gamma y^2},\\
\!\!&&\!\! F^{ij}y_j=f_{1}x^i+f_{2}\frac{y^i}{r}+f_{3}a^i.
\end{eqnarray*}
Since $\omega,L^2,b_{1},b_{2},a^2,d_{1},d_{2},d_{3},M^2,e_{1},e_{2},e_{3},N^2,f_{1},f_{2},f_{3}$ and $y^2$ are too long, they are listed in Appendix of \cite{RG}. Taking a vertical derivation of fundamental tensor gives us the Cartan torsion. In local coordinates, the Cartan torsion is given by
\begin{align}\label{1Cx}
C_{ijk}:=\frac{1}{2}\frac{\partial g_{ij}}{\partial y^k}.
\end{align}
Let $F=r\phi(u,s,v,t)$ be a general spherically symmetric Finsler metric in $\mathbb{R}^n$. By putting \eqref{eq8} into \eqref{1Cx}, one can get 
\begin{eqnarray}\label{eq9}
\nonumber C_{ijk}\!\!&=&\!\!\frac{1}{2r}\Big[\big(\phi\phi_{s}-s(\phi_{s}^2+\phi\phi_{ss})-t(\phi_{s}\phi_{t}+\phi\phi_{ts})\big)x^i\delta_{jk}
+\big(\phi\phi_{t}-s(\phi_{s}\phi_{t}+\phi\phi_{st})\\
\nonumber\!\!&&\!\! -t(\phi_{t}^2+\phi\phi_{tt})\big)a^i\delta_{jk}\Big](i\rightarrow j\rightarrow k\rightarrow i)+\frac{1}{2r^2}\Big[\big(s^2(\phi_{s}^2+\phi\phi_{ss})+t^2(\phi_{t}^2+\phi\phi_{tt})\\
\nonumber\!\!&&\!\! +2ts(\phi_{s}\phi_{t}+\phi\phi_{ts})-s\phi\phi_{s}-t\phi\phi_{t}\big)y^i\delta_{jk}\Big](i\rightarrow j\rightarrow k\rightarrow i)\\
\nonumber\!\!&&\!\! +\frac{1}{2r}\Big[(2\phi_{t}\phi_{ts}+\phi_{s}\phi_{tt}+\phi\phi_{tts})x^ia^ja^k
+(\phi_{t}\phi_{ss}+2\phi_{s}\phi_{ts}+\phi\phi_{sst})x^ix^ja^k\Big](i\rightarrow j\rightarrow k\rightarrow i)\\
\nonumber\!\!&&\!\! +\frac{1}{2r}\big(3\phi_{t}\phi_{tt}+\phi\phi_{ttt}\big)a^ia^ja^k-\frac{1}{2r^2}\big(s(2\phi_{t}\phi_{ts}+\phi_{s}\phi_{tt}+\phi\phi_{tts})\\
\nonumber\!\!&&\!\! +t(3\phi_{t}\phi_{tt}+\phi\phi_{ttt})\big)y^ia^ja^k(i\rightarrow j\rightarrow k\rightarrow i)
+\frac{1}{2r^3}\big(s^2(2\phi_{s}\phi_{st}+\phi_{t}\phi_{ss}+\phi\phi_{sst})\\
\nonumber\!\!&&\!\! +t^2(3\phi_{t}\phi_{tt}+\phi\phi_{ttt})+2ts(2\phi_{t}\phi_{st}
+\phi_{s}\phi_{tt}+\phi\phi_{stt})+s(\phi_{t}\phi_{s}+\phi\phi_{st})+t(\phi_{t}^2+\phi\phi_{tt})\\
\nonumber\!\!&&\!\! -\phi\phi_{t}\big)y^iy^ja^k(i\rightarrow j\rightarrow k\rightarrow i)-\frac{1}{2r^4}
\big(s^3(3\phi_{s}\phi_{ss}+\phi\phi_{sss})+t^3(3\phi_{t}\phi_{tt}+\phi\phi_{ttt})\\
\nonumber\!\!&&\!\! +3s^2(\phi_{s}^2+\phi\phi_{ss})+3t^2(\phi_{t}^2+\phi\phi_{tt})+3t^2s(2\phi_{t}\phi_{st}
+\phi_{s}\phi_{tt}+\phi\phi_{stt})\\
\nonumber\!\!&&\!\! +3s^2t(\phi_{t}\phi_{ss}+2\phi_{s}\phi_{ts}+\phi\phi_{sst})+6ts(\phi_{s}\phi_{t}+\phi\phi_{st})
-3s\phi\phi_{s}-3t\phi\phi_{t}\big)y^iy^jy^k\\
\nonumber\!\!&&\!\!+\frac{1}{2r^3}\big(s^2(3\phi_{s}\phi_{ss}+\phi\phi_{sss})+t^2(2\phi_{t}\phi_{ts}+\phi_{s}\phi_{tt}
+\phi\phi_{tts})+2ts(\phi_{t}\phi_{ss}+2\phi_{s}\phi_{ts}+\phi\phi_{sst})\\
\nonumber\!\!&&\!\! +s(\phi_{s}^2+\phi\phi_{ss})+t(\phi_{s}\phi_{t}+\phi\phi_{st})-\phi\phi_{s}\big)x^iy^jy^k-\frac{1}{2r^2}\Big[\big(s(3\phi_{s}\phi_{ss}+\phi\phi_{sss})\\
\nonumber\!\!&&\!\! +t(\phi_{t}\phi_{ss}+2\phi_{s}\phi_{ts}+\phi\phi_{sst})\big)x^ix^jy^k+\big(s(\phi_{t}\phi_{ss}+2\phi_{s}\phi_{st}
+\phi\phi_{sst})\\
\nonumber\!\!&&\!\! +t(2\phi_{t}\phi_{ts}+\phi_{s}\phi_{tt}+\phi\phi_{tts})\big)(x^ia^jy^k+x^iy^ja^k)\Big](i\rightarrow j\rightarrow k\rightarrow i)\\
\!\!&&\!\! +\frac{1}{2r}\big(3\phi_{s}\phi_{ss}+\phi\phi_{sss}\big)x^ix^jx^k,
\end{eqnarray}
where $i\rightarrow j\rightarrow k\rightarrow i$ denotes cyclic permutation.

\bigskip
For a general spherically symmetric Finsler metric $F=r\phi(u,s,v,t)$, let us put
\begin{equation}\label{eq92}
\begin{split}
r:=&|y|,\quad u:=\vert x\vert ^2,\quad s:=\frac{\langle x,y\rangle}{\vert y\vert},\quad v:=\langle a,x\rangle ,\quad t:=\frac{\langle a,y\rangle}{|y|},\\
r^i:=&r_i:=\frac{y^i}{r},\quad x^i:=x_i,\quad s^i:=s_i:=x_i-sr_i,\quad t^i:=t_i:=a_i-tr_i.
\end{split}
\end{equation}
Then, we have the following.
\begin{lem}\label{eq106}
Let $F=r\phi(u,s,v,t)$ be a general spherically symmetric Finsler metric on domain $\mathcal{U}\subseteq\mathbb{R}^n$.Then the following holds
\begin{align}\label{eq96}
g^{ij}s_is_j=c_0^{-1}&\Big[(u-s^2)-\mathfrak{A}(u-s^2)^2-2\mathfrak{B}(u-s^2)(v-st)-\mathfrak{C}(v-st)^2\Big],\\
\label{eq97}g^{ij}s_it_j=c_0^{-1}&\Big[(v-st)-\mathfrak{A}(u-s^2)(v-st)-\mathfrak{B}\big((v-st)^2+(u-s^2)(a^2-t^2)\big)\notag\\
&-\mathfrak{C}(v-st)(a^2-t^2)\Big],\\
\label{eq98}g^{ij}t_it_j=c_0^{-1}&\Big[(a^2-t^2)-\mathfrak{A}(v-st)^2-2\mathfrak{B}(a^2-t^2)(v-st)-\mathfrak{C}(a^2-t^2)^2\Big],
\end{align}
where
\begin{align*}
c_0&=\phi ^2-s\phi\phi _s -t\phi\phi _t,\\
\mathfrak{A}&=\zeta+\tau\omega^2+\nu b_1^2+\sigma d_1^2+\kappa e_1^2+\alpha f_1^2,\\
\mathfrak{B}&=\nu b_1+\sigma d_1d_3+\kappa e_1e_3+\alpha f_1f_3,\\
\mathfrak{C}&=\nu +\sigma d_3^2+\kappa e_3^2+\alpha f_3^2.
\end{align*}
\end{lem}
\begin{proof}
It  is easy to see that the following hold
\begin{align}\label{eq93}
r^ir_i=1,\quad x^ir_i=s, \quad a^ir_i=t,\quad r_iy^i=r,\quad r_is^i=0,\quad r_it^i=0.
\end{align}
 Contracting \eqref{eq91} with $s_j$ and using \eqref{eq92} and \eqref{eq93} in it implies that
\begin{align}\label{eq94}
g^{ij}s_j=c_0^{-1}&\Big[(x^i-sr^i)-(u-s^2)\big(\zeta x^i+\tau L^i\omega +\nu(B^{ij}a_j)b_1+\sigma M^id_1+\kappa N^ie_1\notag\\
&+\alpha (F^{ij}y_j)f_1\big)-(v-st)\big(\nu(B^{ij}a_j)+\sigma M^id_3+\kappa N^ie_3+\alpha (F^{ij}y_j)f_3\big)\Big].
\end{align}
Multiplying \eqref{eq94} with $s_i$ yields
\begin{align}\label{eq95}
g^{ij}s_is_j=c_0^{-1}&\Big[(x^i-sr^i)-(u-s^2)\big(\zeta x^i+\tau L^i\omega +\nu(B^{ij}a_j)b_1+\sigma M^id_1+\kappa N^ie_1+\alpha (F^{ij}y_j)f_1\big)\notag\\
&-(v-st)\big(\nu(B^{ij}a_j)+\sigma M^id_3+\kappa N^ie_3+\alpha (F^{ij}y_j)f_3\big)\Big](x_i-sr_i).
\end{align}
By putting \eqref{eq92} and \eqref{eq93} in \eqref{eq95}, we get \eqref{eq96}. By a similar method, one can  obtain \eqref{eq97} and \eqref{eq98}.
\end{proof}

\bigskip

\begin{lem}\label{I}
Let $F=r\phi(u,s,v,t)$ be a general spherically symmetric Finsler metric on a domain $\mathcal{U}\subseteq\mathbb{R}^n$. Then the norm of mean Cartan torsion of $F$ is given by following
\begin{align}\label{eq105}
\|\mathbf{I}\|^2=\frac{1}{4r^2}\Big\lbrace\|\mathfrak{S}\|^2 H^2+2\|\mathfrak{R}\|^2 HK+\|\mathfrak{T}\|^2 K^2\Big\rbrace ,
\end{align}
where
\[
\|\mathfrak{S}\|^2:=g^{ij}s_is_j,\quad \|\mathfrak{R}\|^2:=g^{ij}s_it_j,\quad \|\mathfrak{T}\|^2:=g^{ij}t_it_j
\]
and
\begin{align*}
H:=&(n+1)\Big(\frac{\phi_s}{\phi}+\frac{\sigma_{1s}}{\sigma_1}\Big)+\|\mathfrak{S}\|^2\Big(\phi_{sss}
-\frac{3\sigma_{1s}\phi_{ss}}{\sigma_1}\Big)\phi
+2\|\mathfrak{R}\|^2\Big(\phi_{sst}-\frac{2\sigma_{1s}\phi_{st}+\sigma_{1t}\phi_{ss}}{\sigma_1}\Big)\phi\\
&+\|\mathfrak{T}\|^2\Big(\phi_{tts}-\frac{2\sigma_{1t}\phi_{ts}+\sigma_{1s}\phi_{tt}}{\sigma_1}\Big)\phi,\\
K:=&(n+1)\Big(\frac{\phi_t}{\phi}+\frac{\sigma_{1t}}{\sigma_1}\Big)+\|\mathfrak{T}\|^2\Big(\phi_{ttt}-\frac{3\sigma_{1t}\phi_{tt}}{\sigma_1}\Big)\phi +\|\mathfrak{S}\|^2\Big(\phi_{sst}-\frac{2\sigma_{1s}\phi_{st}+\sigma_{1t}\phi_{ss}}{\sigma_1}\Big)\phi\\
&+2\|\mathfrak{R}\|^2\Big(\phi_{tts}-\frac{2\sigma_{1t}\phi_{ts}+\sigma_{1s}\phi_{tt}}{\sigma_1}\Big)\phi.
\end{align*}
Here, $\sigma_1:=\phi -s\phi_s -t\phi_t$.
\end{lem}
\begin{proof}
By simple calculations, we get
\begin{align}
s_iy^i=0,\quad s_is^i=u-s^2,\quad t_iy^i=0,\quad t_it^i=a^2-t^2,\quad s_it^i=v-st.
\end{align}
A direct computation shows that the angular metric of $F$ is given by the following
\[
h_{ij}=\sigma_1(\delta_{ij}-r_ir_j)\phi+\phi\phi_{ss}s_is_j+\phi\phi_{tt}t_it_j+\phi\phi_{st}(s_it_j+s_jt_i).
\]
\eqref{eq8} can be rewritten in the following simple form
\begin{equation}
\begin{split}\label{eq99}
g_{ij}=&\sigma_1\phi\delta_{ij}+(s\phi\phi_s+t\phi\phi_t)r_ir_j+\phi\phi_s(s_ir_j+s_jr_i)+(\phi_s^2+\phi\phi_{ss})s_is_j\\
&+\phi\phi_t(t_ir_j+t_jr_i)+(\phi_t^2+\phi\phi_{tt})t_it_j+(\phi_s\phi_t+\phi\phi_{st})(s_it_j+s_jt_i).
\end{split}
\end{equation}
Taking a vertical derivation  of \eqref{eq99} yields
\begin{equation}
\begin{split}\label{eq100}
2rC_{ijk}=&\frac{\sigma_2}{\sigma_1\phi}\Big\lbrace h_{ij}s_k+h_{jk}s_i+h_{ki}s_j\Big\rbrace+\frac{1}{\sigma_1}\Big\lbrace \sigma_1\sigma_4-3\sigma_2\phi_{ss}\Big\rbrace s_is_js_k\\
&+\frac{\sigma_3}{\sigma_1\phi}\Big\lbrace h_{ij}t_k+h_{jk}t_i+h_{ki}t_j\Big\rbrace+\frac{1}{\sigma_1}\Big\lbrace \sigma_1\sigma_5-3\sigma_3\phi_{tt}\Big\rbrace t_it_jt_k\\
&+\frac{1}{\sigma_1}\Big\lbrace \sigma_1\sigma_6-2\sigma_2\phi_{st}-\sigma_3\phi_{ss}\Big\rbrace s_is_jt_k\ \  (i\rightarrow j\rightarrow k\rightarrow i)\\
&+\frac{1}{\sigma_1}\Big\lbrace \sigma_1\sigma_7-2\sigma_3\phi_{ts}-\sigma_2\phi_{tt}\Big\rbrace  t_it_js_k\ \ (i\rightarrow j\rightarrow k\rightarrow i),
\end{split}
\end{equation}
where
\begin{align*}
\sigma_2&:=(\phi-s\phi_s-t\phi_t)\phi_s-s\phi\phi_{ss}-t\phi\phi_{ts},\\
\sigma_3&:=(\phi-s\phi_s-t\phi_t)\phi_t-s\phi\phi_{st}-t\phi\phi_{tt},\\
\sigma_4&:=3\phi_s\phi_{ss}+\phi\phi_{sss},\\\\
\sigma_5&:=3\phi_t\phi_{tt}+\phi\phi_{ttt},\\
\sigma_6&:=2\phi_s\phi_{st}+\phi_t\phi_{ss}+\phi\phi_{sst},\\
\sigma_7&:=2\phi_t\phi_{ts}+\phi_s\phi_{tt}+\phi\phi_{tts}.
\end{align*}
Multiplying \eqref{eq100} with $\sigma_1\phi$ give us
\begin{equation}
\begin{split}\label{eq101}
2\sigma_1 r\phi C_{ijk}=&\sigma_2\Big\lbrace h_{ij}s_k+h_{jk}s_i+h_{ki}s_j\Big\rbrace+\phi\Big\lbrace \sigma_1\sigma_4-3\sigma_2\phi_{ss}\Big\rbrace s_is_js_k\\
&+\sigma_3\Big\lbrace h_{ij}t_k+h_{jk}t_i+h_{ki}t_j\Big\rbrace+\phi\Big\lbrace \sigma_1\sigma_5-3\sigma_3\phi_{tt}\Big\rbrace t_it_jt_k\\
&+\phi\Big\lbrace\sigma_1\sigma_6-2\sigma_2\phi_{st}-\sigma_3\phi_{ss}\Big\rbrace s_is_jt_k \ \ (i\rightarrow j\rightarrow k\rightarrow i)\\
&+\phi\Big\lbrace \sigma_1\sigma_7-2\sigma_3\phi_{ts}-\sigma_2\phi_{tt}\Big\rbrace t_it_js_k \ \ (i\rightarrow j\rightarrow k\rightarrow i).
\end{split}
\end{equation}
Let us put
\begin{align*}
\sigma_8&:=\phi(\sigma_1\sigma_4-3\sigma_2\phi_{ss}),\\
\sigma_9&:=\phi(\sigma_1\sigma_5-3\sigma_3\phi_{tt}),\\
\sigma_{10}&:=\phi(\sigma_1\sigma_6-2\sigma_2\phi_{st}-\sigma_3\phi_{ss}),\\
\sigma_{11}&:=\phi(\sigma_1\sigma_7-2\sigma_3\phi_{ts}-\sigma_2\phi_{tt}),\\
\sigma_{12}&:=r\phi(\phi-s\phi_s-t\phi_t).
\end{align*}
Therefore, \eqref{eq101} can be written as follows
\begin{equation}
\begin{split}
2\sigma_{12}C_{ijk}=&\sigma_2\Big\lbrace h_{ij}s_k+h_{jk}s_i+h_{ki}s_j\Big\rbrace+\sigma_3\Big\lbrace h_{ij}t_k+h_{jk}t_i+h_{ki}t_j\Big\rbrace+\sigma_8 s_is_js_k\\
&+\sigma_9 t_it_jt_k+(\sigma_{10}s_is_jt_k+\sigma_{11} t_it_js_k)(i\rightarrow j\rightarrow k\rightarrow i),
\end{split}
\end{equation}
or equivalently
\begin{equation}
\begin{split}\label{eq102}
C_{ijk}=&\frac{\sigma_{2}}{2\sigma_{12}}\Big\lbrace h_{ij}s_k+h_{jk}s_i+h_{ki}s_j\Big\rbrace+\frac{\sigma_3}{2\sigma_{12}}\Big\lbrace h_{ij}t_k+h_{jk}t_i+h_{ki}t_j\Big\rbrace+\frac{\sigma_8}{2\sigma_{12}} s_is_js_k\\
&+\frac{\sigma_9}{2\sigma_{12}}t_it_jt_k+\frac{1}{2\sigma_{12}}\Big(\sigma_{10}s_is_jt_k+\sigma_{11} t_it_js_k\Big)(i\rightarrow j\rightarrow k\rightarrow i).
\end{split}
\end{equation}
Taking a trace of \eqref{eq102} yields
\begin{align}\label{eq103}
I_k=&\bigg[\frac{(n+1)\sigma_2+ \|\mathfrak{S}\|^2\sigma_8 +2\|\mathfrak{R}\|^2\sigma_{10}+\|\mathfrak{T}\|^2\sigma_{11}}{2\sigma_{12}}\bigg]s_k\notag\\
&+\bigg[\frac{(n+1)\sigma_3 +\|\mathfrak{T}\|^2\sigma_9 +\|\mathfrak{S}\|^2\sigma_{10}+2\|\mathfrak{R}\|^2\sigma_{11}}{2\sigma_{12}}\bigg]t_k.
\end{align}
Now, by using \eqref{eq103} and after straightforward calculations, we get
\begin{align}\label{HK}
I_k=\frac{1}{2r}\big(Hs_k+Kt_k\big),
\end{align}
where
\begin{eqnarray*}
H\!\!&:=&\!\! (n+1)\Big(\frac{\phi_s}{\phi}+\frac{\sigma_{1s}}{\sigma_1}\Big)+\|\mathfrak{S}\|^2\Big(\phi_{sss}-\frac{3\sigma_{1s}\phi_{ss}}{\sigma_1}\Big)\phi
+2\|\mathfrak{R}\|^2\Big(\phi_{sst}-\frac{2\sigma_{1s}\phi_{st}+\sigma_{1t}\phi_{ss}}{\sigma_1}\Big)\phi\\
\!\!&&\!\! +\|\mathfrak{T}\|^2\Big(\phi_{tts}-\frac{2\sigma_{1t}\phi_{ts}+\sigma_{1s}\phi_{tt}}{\sigma_1}\Big)\phi ,\\
K\!\!&:=&\!\!(n+1)\Big(\frac{\phi_t}{\phi}+\frac{\sigma_{1t}}{\sigma_1}\Big)+\|\mathfrak{T}\|^2\Big(\phi_{ttt}-\frac{3\sigma_{1t}\phi_{tt}}{\sigma_1}\Big)\phi +\|\mathfrak{S}\|^2\Big(\phi_{sst}-\frac{2\sigma_{1s}\phi_{st}+\sigma_{1t}\phi_{ss}}{\sigma_1}\Big)\phi\\
\!\!&&\!\! +2\|\mathfrak{R}\|^2\Big(\phi_{tts}-\frac{2\sigma_{1t}\phi_{ts}+\sigma_{1s}\phi_{tt}}{\sigma_1}\Big)\phi.
\end{eqnarray*}
By \eqref{HK}, we get
\begin{align}\label{IbK}
I^k=\frac{1}{2rc_0}&\bigg\lbrace\Big[\big(1-\mathfrak{A}(u-s^2)-\mathfrak{B}(v-st)\big)H-\big(\mathfrak{A}(v-st)+\mathfrak{B}(a^2-t^2)\big)K\Big]x^k\notag\\
& -\Big[\big(s+\mathfrak{D}(u-s^2)+\mathfrak{E}(v-st)\big)H+\big(t+\mathfrak{D}(v-st)+\mathfrak{E}(a^2-t^2)\big)K\Big]r^k\notag\\
& -\Big[\big(\mathfrak{B}(u-s^2)+\mathfrak{C}(v-st)\big)H-\big(1-\mathfrak{B}(v-st)-\mathfrak{C}(a^2-t^2)\big)K\Big]a^k\bigg\rbrace ,
\end{align}
where $\mathfrak{A}, \mathfrak{B}$ and $\mathfrak{C}$ defined in Lemma \ref{eq106}. Moreover
\begin{align*}
\mathfrak{D}:=&\tau\omega +\nu b_1b_2+\sigma d_1d_2+\kappa e_1e_2+\alpha f_1f_2,\\
\mathfrak{E}:=&\nu b_2+\sigma d_2d_3+\kappa e_2e_3+\alpha f_2f_3.
\end{align*}
By definition, we have $\|\mathbf{I}\|^2 =I^kI_k$. Then, by multiplying \eqref{HK} with \eqref{IbK} we get \eqref{eq105}.
\end{proof}

\bigskip
For a spherically symmetric Finsler metric $F=u\phi(r,s)$, we conclude the following.

\begin{cor}
Let $F=u\phi(r,s)$ be a spherically symmetric Finsler metric on a domain $\mathcal{U}\subseteq\mathbb{R}^n$. Then the norm of mean Cartan torsion of $F$ is given by following
\begin{eqnarray}\label{120}
\Vert\mathbf{I}\Vert=\bigg\vert\frac{(n+1)(\sigma_1\phi_s-s\phi\phi_{ss})(\sigma_1 +(r^2-s^2)\phi_{ss})+(r^2-s^2)(\sigma_1\phi_{sss}+3s\phi_{ss}^2)\phi}
{2u\sigma_1\rho}\bigg\vert\sqrt{\frac{r^2-s^2}{\rho}},\ \
\end{eqnarray}
where $\rho:=\phi[\sigma_1+(r^2-s^2)\phi_{ss}]$.
\end{cor}
\begin{proof}
For a spherically symmetric Finsler metric $F=u\phi(r,s)$, we have $a=0$ which yields $v=t=0$. By putting these in \eqref{eq105}, we get
\begin{align}\label{eq114}
\Vert\mathbf{I}\Vert^2=\frac{1}{4u^2}\|\mathfrak{S}\|^2 H^2,
\end{align}
where
\begin{align*}
H:=(n+1)\big(\frac{\phi_s}{\phi}+\frac{\sigma_{1s}}{\sigma_1}\big)+\|\mathfrak{S}\|^2\big(\phi_{sss}
-\frac{3\sigma_{1s}\phi_{ss}}{\sigma_1}\big)\phi.
\end{align*}
 Moreover, $\|\mathfrak{S}\|^2=(r^2-s^2)-\mathfrak{A}(r^2-s^2)$
where
\begin{align*}
\mathfrak{A}:=\zeta+\tau\omega^2+\nu b_1^2+\sigma d_1^2+\kappa e_1^2+\alpha f_1^2.
\end{align*}
On the other hand by using \eqref{eq8}, \eqref{eq91} and simple calculations, we get
\begin{align}\label{eq107}
\zeta &=\frac{\phi_s}{s\phi +(r^2-s^2)\phi_s},\\
\label{eq108}\omega &=-\frac{\phi}{s\phi +(r^2-s^2)\phi_s},\\
\label{eq109}\tau &=\frac{\big[s(\phi_s ^2+\phi\phi_{ss})-\phi\phi_s\big]\big[s\phi +(r^2-s^2)\phi_s\big]}{\phi^2\big[\phi -s\phi_s +(r^2-s^2)\phi_{ss}\big]},\\
\label{eq110}\nu &=\sigma =\kappa =\alpha =0.
\end{align}
Therefore by putting \eqref{eq107}-\eqref{eq110} in $\mathfrak{A}$, we have
\begin{align}\label{eq111}
\mathfrak{A}=\frac{\phi_{ss}}{\phi -s\phi_s +(r^2-s^2)\phi_{ss}}.
\end{align}
Now, using straightforward simplifications, we have
\begin{align}\label{eq112}
H=\frac{(n+1)\big[(\phi -s\phi_s)\phi_s -s\phi\phi_{ss}\big]\big[\phi -s\phi_s +(r^2-s^2)\phi_{ss}\big]+(r^2-s^2)\big[(\phi -s\phi_s)\phi_{sss}+3s\phi_{ss}^2\big]\phi}{\phi(\phi -s\phi_s)\big[\phi -s\phi_s +(r^2-s^2)\phi_{ss}\big]}.
\end{align}
Substituting \eqref{eq111} and \eqref{eq112} into \eqref{eq114} give us the proof.
\end{proof}
\begin{ex}
Let
\begin{align*}
F=\frac{\big(\sqrt{|y|^2-(|x|^2|y|^2-\langle x,y\rangle ^2)}+\langle x,y\rangle\big)^2}{(1-|x|^2)^2\sqrt{|y|^2-(|x|^2|y|^2-\langle x,y\rangle ^2)}},
\end{align*}
be the Berwald metric, where $y\in T_x\mathbb{B}^2\cong\mathbb{R}^2$.

It is clear that the corresponding function $\phi(r,s)$ of a spherically symmetric Finsler metric $F=|y|\phi$ is given by
\begin{align}\label{PHI}
\phi =\frac{(\sqrt{1-r^2+s^2}+s)^2}{(1-r^2)^2\sqrt{1-r^2+s^2}},
\end{align}
where $r=|x|$, $s=\langle x,y\rangle/|y|$ and $|s|<r<1$. It follows from \eqref{120} and \eqref{PHI} that
\begin{align}\label{04}
\Vert\mathbf{I}\Vert =\frac{3}{F}\bigg\vert 1-\frac{s(s+\sqrt{1-r^2+s^2})}{(1-r^2)\big(1+2(r^2-s^2)\big)}\bigg\vert\sqrt{\frac{r^2-s^2}{1+2(r^2-s^2)}}.
\end{align}
By straightforward calculations, we have
\begin{align}\label{01}
-\frac{s(s+\sqrt{1-r^2+s^2})}{1-r^2}\leq\frac{1+s^2}{2(1+\sqrt{1-r^2})}.
\end{align}
Moreover, we get
\begin{align}\label{02}
0<\frac{1}{1+2(r^2-s^2)}<1,\qquad \frac{1+s^2}{1+\sqrt{1-r^2}}>0.
\end{align}
Since $|s|<r$ and by \eqref{01} and \eqref{02}, we obtain
\begin{align}\label{03}
-\frac{s(s+\sqrt{1-r^2+s^2})}{(1-r^2)\big(1+2(r^2-s^2)\big)}<\frac{1+r^2}{2(1+\sqrt{1-r^2})}.
\end{align}
Assume that $F(y)=1$. Then by \eqref{04}, \eqref{03} and 
\begin{align*}
\frac{r^2-s^2}{1+2(r^2-s^2)}<1,
\end{align*}
one can find that the mean Cartan torsion satisfies
\begin{align*}
\Vert\mathbf{I}\Vert <3(1+\frac{1+|x|^2}{2(1+\sqrt{1-|x|^2})}).
\end{align*}
\end{ex}

\bigskip
For a Randers metric $F=\alpha +\beta$, the following is proved.
\begin{prop}\label{Ra} {\rm (\cite{ZSh})} \emph{For any Randers metric $F=\alpha +\beta$  on a $n$-dimensional manifold $M$, where $\alpha$ is a Riemannian metric and $\beta$ is an 1-form on $M$ with $\Vert\beta\Vert_\alpha<1$. The mean Cartan torsion is uniformly bounded. More precisely}
\begin{align*}
\Vert\mathbf{I}\Vert\leq\frac{n+1}{\sqrt{2}}\sqrt{1-\sqrt{1-\Vert\beta\Vert^2}}.
\end{align*}
\end{prop}

\bigskip

\begin{ex}
Let
\begin{align*}
F=\frac{\sqrt{|y|^2-(|x|^2|y|^2-\langle x,y\rangle ^2)}+\langle x,y\rangle}{1-|x|^2},\qquad y\in T_x\mathbb{B}^2\cong\mathbb{R}^2.
\end{align*}
be the Funk metric on $\mathbb{B}^2$. Obviously $F$ is a Randers metric with $|x|<1$. Then by using Proposition \ref{Ra}, we have
\begin{align*}
\Vert\mathbf{I}\Vert\leq\frac{3}{\sqrt{2}}\sqrt{1-\sqrt{1-|x|^2}}.
\end{align*}
\end{ex}

\bigskip

Now, we are ready to prove Theorem \ref{THM1}.

\bigskip

\noindent
{\bf Proof of Theorem \ref{THM1}:}
Contracting \eqref{eq102} by $\delta^{ij}$, we have
\begin{equation}
\begin{split}\label{C0}
C_{iik}=&\ \frac{\sigma_2}{2\sigma_{12}}\big\lbrace h_{ii}s_k+2h_{ik}s_i\big\rbrace +\frac{\sigma_3}{2\sigma_{12}}\big\lbrace h_{ii}t_k+2h_{ik}t_i\big\rbrace+\frac{\sigma_8}{2\sigma_{12}}s_is_is_k+\frac{\sigma_9}{2\sigma_{12}}t_it_it_k\\
&\ +\frac{\sigma_{10}}{2\sigma_{12}}\big(s_is_it_k+2t_is_is_k\big)+\frac{\sigma_{11}}{2\sigma_{12}}\big(t_it_is_k+2s_it_it_k\big).
\end{split}
\end{equation}
Then
\begin{equation}
\begin{split}\label{C1}
C_{iik}=&\ \frac{\sigma_2}{2\sigma_{12}}\big\lbrace h_{ik}s_i+2h_{ki}s_i\big\rbrace +\frac{\sigma_3}{2\sigma_{12}}\big\lbrace h_{ik}t_i+2h_{ki}t_i\big\rbrace+\frac{\sigma_8}{2\sigma_{12}}s_is_is_k+\frac{\sigma_9}{2\sigma_{12}}t_it_it_k\\
&\ +\frac{\sigma_{10}}{2\sigma_{12}}\big(s_it_is_k+2s_is_it_k\big)+\frac{\sigma_{11}}{2\sigma_{12}}\big(t_is_it_k+2t_it_is_k\big).
\end{split}
\end{equation}
 By using \eqref{C0} and \eqref{C1}, we get
 \begin{align}\label{115}
 \sigma_2 h_{ii}s_k+\sigma_3 h_{ii}t_k+\sigma_{10}s_it_is_k+\sigma_{11}t_is_it_k=\sigma_2 h_{ik}s_i+\sigma_3 h_{ik}t_i+\sigma_{10}s_is_it_k+\sigma_{11}t_it_is_k.
 \end{align}
 \eqref{115} is equal to following
 \begin{align}\label{MN}
 Ms_k=Nt_k,
 \end{align}
 where
 \begin{align*}
 M:=&\ \sigma_2\big[\sigma_1(n-2)\phi +(a^2-t^2)\phi\phi_{tt}+(v-st)\phi\phi_{st}\big]-\sigma_3\big[(v-st)\phi\phi_{ss}+(a^2-t^2)\phi\phi_{st}\big]\\
 &\ +\sigma_{10}(v-st)-\sigma_{11}(a^2-t^2),\\
 N:=&\ \sigma_2\big[(v-st)\phi\phi_{tt}+(u-s^2)\phi\phi_{st}\big]-\sigma_3\big[\sigma_1(n-2)\phi +(u-s^2)\phi\phi_{ss}+(v-st)\phi\phi_{st}\big]\\
 &\ +\sigma_{10}(u-s^2)-\sigma_{11}(v-st).
 \end{align*}
 By \eqref{HK} and \eqref{MN}, we have
 \begin{align}\label{116}
 s_k=2r\frac{N}{HN+KM}I_k,\qquad t_k=2r\frac{M}{HN+KM}I_k.
 \end{align}
 Putting \eqref{116} in \eqref{eq102} give us
 \begin{equation}
 \begin{split}\label{117}
 C_{ijk}=&\frac{r\big(\sigma_2 N+\sigma_3 M\big)}{\sigma_{12}\big(HN+KM\big)}\bigg\lbrace h_{ij}I_k+h_{jk}I_i+h_{ki}I_j\bigg\rbrace\\
&+\frac{r^3\big(4\sigma_8 N^3+4\sigma_9 M^3+12\sigma_{10}N^2M+12\sigma_{11}NM^2\big)}{\sigma_{12}\big(HN+KM\big)^3}I_iI_jI_k.
 \end{split}
 \end{equation}
 Thus,  by\eqref{117} we get
 \begin{align}\label{122x}
C_{ijk}=\frac{\mathcal{P}}{1+n}\Big\lbrace h_{ij}I_k+h_{jk}I_i+h_{ki}I_j\Big\rbrace +\frac{\mathcal{Q}}{\Vert\mathbf{I}\Vert^2}I_iI_jI_k,
\end{align}
where
\begin{align*}
\mathcal{P}:=\frac{r(n+1)\big[\sigma_2 N+\sigma_3 M\big]}{\sigma_{12}\big[HN+KM\big]},\qquad
\mathcal{Q}:=\frac{r^3\big[4\sigma_8 N^3+4\sigma_9 M^3+12\sigma_{10}N^2M+12\sigma_{11}NM^2\big]\Vert\mathbf{I}\Vert^2}{\sigma_{12}\big[HN+KM\big]^3}.
\end{align*}
By \eqref{122x}, it follows that every general spherically symmetric Finsler metric $F=r\phi(u,s,v,t)$ is semi-C-reducible.
\qed

\bigskip

For  spherically symmetric Finsler metric, we conclude the following.
\begin{cor}
Let $F=u\phi(r,s)$ be a spherically symmetric Finsler metric on a domain  $\mathcal{U}\subseteq\mathbb{R}^n$. Then the Cartan torsion of $F$ can be written as follows
\begin{align}\label{122}
C_{ijk}=\frac{\mathcal{P}}{1+n}\Big\lbrace h_{ij}I_k+h_{jk}I_i+h_{ki}I_j\Big\rbrace +\frac{\mathcal{Q}}{\Vert\mathbf{I}\Vert^2}I_iI_jI_k,
\end{align}
where
\begin{align*}
\mathcal{P}&:=\frac{(n+1)\big[(\phi -s\phi_s)\phi_s -s\phi\phi_{ss}\big]\big[\phi -s\phi_s +(r^2-s^2)\phi_{ss}\big]}{(n+1)\big[(\phi -s\phi_s)\phi_s -s\phi\phi_{ss}\big]\big[\phi -s\phi_s +(r^2-s^2)\phi_{ss}\big]+(r^2-s^2)\big[(\phi -s\phi_s)\phi_{sss}+3s\phi_{ss}^2\big]\phi},\\
\mathcal{Q}&:=\frac{(r^2-s^2)\big[(\phi-s\phi_s)\phi_{sss}+3s\phi_{ss}^2\big]\phi^2}{(n+1)\big[(\phi -s\phi_s)\phi_s -s\phi\phi_{ss}\big]\big[\phi -s\phi_s +(r^2-s^2)\phi_{ss}\big]\phi +(r^2-s^2)\big[(\phi -s\phi_s)\phi_{sss}+3s\phi_{ss}^2\big]\phi^2}.
\end{align*}
\end{cor}
\begin{proof}
By using \eqref{122} together with the fact that $a=0$ in \eqref{eq3}, we have
\begin{align}\label{121}
C_{ijk}=\frac{\mathcal{P}}{1+n}\Big\lbrace h_{ij}I_k+h_{jk}I_i+h_{ki}I_j\Big\rbrace +\frac{\mathcal{Q}}{\Vert\mathbf{I}\Vert^2}I_iI_jI_k,
\end{align}
where
\begin{align*}
\mathcal{P}:=\frac{(n+1)\big[(\phi -s\phi_s)\phi_s -s\phi\phi_{ss}\big]}{\phi(\phi -s\phi_s)H},\qquad
\mathcal{Q}:=\frac{4u^2\big[(\phi -s\phi_s)\phi_{sss}+3s\phi_{ss}^2\big]\phi^2\|\mathbf{I}\|^2}{\phi(\phi -s\phi_s)H^3}.
\end{align*}
By substituting \eqref{120} and \eqref{eq112} into \eqref{121}, we get \eqref{122}.
\end{proof}
%--------------------------------------------------------------------------------------------------------------------------------------------------------------
\section{Proof Theorem \ref{THM2}}
%--------------------------------------------------------------------------------------------------------------------------------------------------------------

The geodesic spray coefficients of a Finsler metric $F=F(x, y)$ on a manifold $M$ are given by
\begin{align}
G^i:=\frac{1}{4}g^{il}\Big\lbrace\big(F^2\big)_{x^ky^l}y^k-\big(F^2\big)_{x^l}\Big\rbrace =\frac{F_{x^k}y^k}{2F}y^i+\frac{F}{2}g^{il}\Big(F_{x^ky^l}y^k-F_{x^l}\Big).\label{G}
\end{align}
For  a general spherically symmetric Finsler $F=r\phi(u,s,v,t)$ o a domain $\mathcal{U}\subseteq\mathbb{R}^n$, we have
\[
F_{x^k}=2r\phi_{u}x_{k}+\phi_{s}y_{k}+r\phi_{v}a_{k},
\]
one can write the first part as
\begin{align}\label{eq4}
\frac{F_{x^k}y^k}{2F}y^i=\frac{r}{2\phi}\Big(2s\phi_{u}+\phi_{s}+t\phi_{v}\Big)y_{i}.
\end{align}
At the same time, it can be computed that
\begin{eqnarray}
\nonumber F_{x^ky^l}\!\!&=&\!\! \Big(2r\phi_{u}x_{k}+\phi_{s}y_{k}+r\phi_{v}a_{k}\Big)_{y^l}\\
\nonumber \!\!&=&\!\! \frac{2}{r}\Big(\phi_{u}-s\phi_{us}-t\phi_{ut}\Big)x_{k}y_{l}+\frac{1}{r}\Big(\phi_{v}-s\phi_{vs}-t\phi_{vt}\Big)a_{k}y_{l}\\
\nonumber \!\!&&\!\! -\frac{1}{r^2}\Big(s\phi_{ss}+t\phi_{st}\Big)y_{k}y_{l}+\frac{1}{r}\Big(\phi_{st}a_{l}y_{k}+\phi_{ss}x_{l}y_{k}\Big)\\
\!\!&&\!\! +2\Big(\phi_{us}x_{l}x_{k}+\phi_{ut}a_{l}x_{k}\Big)+\phi_{vt}a_{k}a_{l}+\phi_{s}\delta_{lk}+\phi_{vs}a_{k}x_{l}.
\end{eqnarray}
Hence
\begin{equation}\label{eq5}
\begin{split}
F_{x^ky^l}y^k-F_{x^l}=&\ r\Big(2s\phi_{us}+\phi_{ss}+t\phi_{sv}-2\phi_u \Big)\Big(x_l-s\frac{y_{l}}{r}\Big)\\
&\ +r\Big(2s\phi_{ut}+\phi_{st}+t\phi_{vt}-\phi_ v\Big)\Big(a_l-t\frac{y_l}{r}\Big).
\end{split}
\end{equation}
%\end{align}
Putting  \eqref{eq4} and \eqref{eq5} in \eqref{G} implies that
\begin{align}
\nonumber G^i=\frac{r}{2\phi}\big(2s\phi_{u}+\phi_{s}+t\phi_{v}\big)y_{i}+\frac{r^2\phi}{2}g^{il}&\Big\lbrace\big (2s\phi_{us}+\phi_{ss}+t\phi_{sv}-2\phi_{u}\big)\big(x_{l}-s\frac{y_{l}}{r}\big)\\
&+\big(2s\phi_{ut}+\phi_{st}+t\phi_{vt}-\phi_{v}\big)\big(a_{l}-t\frac{y_{l}}{r}\big)\Big\rbrace.\label{G0}
\end{align}
The following hold
\begin{align}
g^{il}\Big(x_{l}-s\frac{y_{l}}{r}\Big)=c_{0}^{-1}\Big[Cx^i+D\frac{y^i}{r}+Ea^i\Big],\label{eq6}\\
g^{il}\Big(a_{l}-t\frac{y_{l}}{r}\Big)=c_{0}^{-1}\Big[Gx^i+H\frac{y^i}{r}+Ia^i\Big],\label{eq7}
\end{align}
where $C,D,E,G,H$ and $I$ are too long, they are listed in Appendix of  \cite{RG}. By putting \eqref{eq6} and \eqref{eq7} into \eqref{G} and simplifying the result it, we get
\begin{align}\label{eq12}
G^i=rPy^i+r^2Qx^i+r^2Ra^i,
\end{align}
where $P$, $Q$ and $R$ being long, they are again listed in Appendix section of \cite{RG}.

By putting \eqref{eq12} into \eqref{L}, we get
\begin{equation}\label{eq48}
\begin{split}
L_{jkl}=&-\dfrac{\phi}{2}\Big[ L_{1} x^jx^kx^l+L_{2}(x^j\delta_{kl}+x^k\delta_{jl}+x^l\delta_{jk})+L_{3}a^ja^ka^l\\
&+L_{4}(a^j\delta_{kl}+a^k\delta_{jl}+a^l\delta_{jk})+L_{5}(x^jx^ka^l+x^kx^la^j+x^lx^ja^k)\\
&+L_{6}(x^ja^ka^l+x^ka^la^j+x^la^ja^k)+L_{7}\dfrac{y^j}{r}\dfrac{y^k}{r}\dfrac{y^l}{r}+L_{8}(\dfrac{y^j}{r}\delta_{kl}+\dfrac{y^k}{r}\delta_{jl}+\dfrac{y^l}{r}\delta_{jk})\\
&+L_{9}(\dfrac{y^j}{r}x^kx^l+\dfrac{y^k}{r}x^lx^j+\dfrac{y^l}{r}x^jx^k)+L_{10}(x^j\dfrac{y^k}{r}\dfrac{y^l}{r}+x^k\dfrac{y^j}{r}\dfrac{y^l}{r}+x^l\dfrac{y^j}{r}\dfrac{y^k}{r})\\
&+L_{11}(\dfrac{y^j}{r}a^ka^l+\dfrac{y^k}{r}a^ja^l+\dfrac{y^l}{r}a^ja^k)+L_{12}(a^j\dfrac{y^k}{r}\dfrac{y^l}{r}+a^k\dfrac{y^j}{r}\dfrac{y^l}{r}+a^l\dfrac{y^j}{r}\dfrac{y^k}{r})\\
&+L_{13}\big(x^j(a^k\dfrac{y^l}{r}+a^l\dfrac{y^k}{r})+x^k(a^l\dfrac{y^j}{r}+a^j\dfrac{y^l}{r})+x^l(a^j\dfrac{y^k}{r}+a^k\dfrac{y^j}{r})\big)\Big],
\end{split}
\end{equation}
where
\begin{eqnarray}
\nonumber L_1\!\!&=&\!\! 3\phi_{s}P_{ss}+\phi P_{sss}+\big(s\phi +(u-s^2)\phi_{s}+(v-st)\phi_{t}\big)Q_{sss}+\big(t\phi +(v-st)\phi_{s}+(a^2-t^2)\phi_{t}\big)R_{sss},\\
\nonumber L_2\!\!&=&\!\! -\phi(sP_{ss}+tP_{st})+\phi_{s}(P-sP_{s}-tP_{t})+\big(s\phi +(u-s^2)\phi_{s}+(v-st)\phi_{t}\big)(Q_{s}-sQ_{ss}-tQ_{st})\\
\nonumber\!\!&&\!\!  +\big(t\phi +(v-st)\phi_{s}+(a^2-t^2)\phi_{t}\big)(R_{s}-sR_{ss}-tR_{st}),\\
\nonumber L_3\!\!&=&\!\! \phi P_{ttt}+3\phi_{t}P_{tt}+\big(s\phi +(u-s^2)\phi_{s}+(v-st)\phi_{t}\big)Q_{ttt}+\big(t\phi +(v-st)\phi_{s}+(a^2-t^2)\phi_{t}\big)R_{ttt},\\
\nonumber L_4\!\!&=&\!\! -\phi(tP_{tt}+sP_{ts})+\phi_{t}(P-sP_{s}-tP_{t})+\big(s\phi +(u-s^2)\phi_{s}+(v-st)\phi_{t}\big)(Q_{t}-tQ_{tt}-tQ_{ts})\\
\nonumber\!\!&&\!\!  +\big(t\phi +(v-st)\phi_{s}+(a^2-t^2)\phi_{t}\big)(R_{t}-tR_{tt}-sR_{ts}),\\
\nonumber L_5\!\!&=&\!\! \phi_{t}P_{ss}+2\phi_{s}P_{st}+\phi P_{sst}+\big(s\phi +(u-s^2)\phi_{s}+(v-st)\phi_{t}\big)Q_{sst}\\
\nonumber \!\!&&\!\!  +\big(t\phi +(v-st)\phi_{s}+(a^2-t^2)\phi_{t}\big)R_{sst},\\
\nonumber L_6\!\!&=&\!\! \phi_{s}P_{tt}+2\phi_{t}P_{ts}+\phi P_{stt}+\big(s\phi +(u-s^2)\phi_{s}+(v-st)\phi_{t}\big)Q_{stt}\\
\nonumber \!\!&&\!\!  +\big(t\phi +(v-st)\phi_{s}+(a^2-t^2)\phi_{t}\big)R_{sst},\\
\nonumber L_7\!\!&=&\!\! -s^3L_{1}+3sL_{2}-t^3L_{3}+3tL_{4}-3s^2tL_{5}-3st^2L_{6},\\
\nonumber L_8\!\!&=&\!\! -sL_{2}-tL_{4},\\
\nonumber L_9\!\!&=&\!\! -sL_{1}-tL_{5},\\
\nonumber L_{10}\!\!&=&\!\! s^2L_{1}-L_{2}+2stL_{5}+t^2L_{6},\\
\nonumber L_{11}\!\!&=&\!\! -tL_{3}-sL_{6},\\
\nonumber L_{12}\!\!&=&\!\! t^2L_{3}-L_{4}+s^2L_{5}+2stL_{6},\\
\nonumber L_{13}\!\!&=&\!\! -sL_{5}-tL_{6}.
\end{eqnarray}
By substitiution \eqref{eq91} and \eqref{eq48} into \eqref{Jg}, $J_i$ is given by
\begin{align}\label{J}
J_i:=H(x^i-s\frac{y^i}{r})+K(a^i-t\frac{y^i}{r}),
\end{align}
where
\begin{align*}
H:=\frac{\phi}{2c_0}\big\lbrace\varphi_0 L_1+\varphi_1 L_2+\varphi_2 L_4+2\varphi_3 L_5+\varphi_4 L_6\big\rbrace,\\
K:=\frac{\phi}{2c_0}\big\lbrace 2\varphi_5 L_2+\varphi_4 L_3+ \varphi_6 L_4+\varphi_0 L_5+2\varphi_3 L_6\big\rbrace ,
\end{align*}
and
\begin{align*}
\varphi_0 &:= -(u-s^2)+A(u-s^2)^2+2B(u-s^2)(v-st)+C(v-st)^2,\\
\varphi_1 &:= -(n+1)+3A(u-s^2)+4B(v-st)+C(a^2-t^2),\\
\varphi_2 &:= A(v-st)+B(a^2-t^2),\\
\varphi_3 &:=-(v-st)+A(u-s^2)(v-st)+B\big((v-st)^2+(u-s^2)(a^2-t^2)\big)+C(v-st)(a^2-t^2),\\
\varphi_4 &:= -(a^2-t^2)+A(v-st)^2+2B(v-st)(a^2-t^2)+C(a^2-t^2)^2,\\
\varphi_5 &:= B(u-s^2)+C(v-st),\\
\varphi_6 &:= -(n+1)+A(u-s^2)+4B(v-st)+3C(a^2-t^2),
\end{align*}
and
\begin{align*}
A &:= \zeta +\tau\omega^2 +\nu b_1^2+\sigma d_1^2+\kappa e_1^2+\alpha f_1^2,\\
B &:= \nu b_1^2+\sigma d_1d_3+\kappa e_1e_3+\alpha f_1f_3,\\
C &:= \nu +\sigma d_3^2+\kappa e_3^2+\alpha f_3^2.
\end{align*}
By \eqref{J}, it is easy to see that $F$ is a weakly Landsberg metric if and only if $H=K=0$.
\begin{cor}
Let $F=u\phi(r,s)$ be a spherically symmetric Finsler metric on a domain  $\mathcal{U}\subseteq\mathbb{R}^n$. Then the mean Landsberg curvature $F$ can be written as follows
\begin{align}
J_i=H(x^i-s\frac{y^i}{u}),
\end{align}
where
\begin{align*}
H:=-\frac{(r^2-s^2)(\phi -s\phi_s)L_1+\big[(n+1)(\phi -s\phi_s)+(n-2)(r^2-s^2)\phi_{ss}\big]L_2}{2(\phi -s\phi_s)\big[\phi -s\phi_s +(r^2-s^2)\phi_{ss}\big]}.
\end{align*}
\end{cor}
\begin{proof}
By \eqref{J} together with the fact that $a=0$ in \eqref{eq3}, we have
\begin{align*}
J_	i=H(x^i-s\frac{y^i}{u}),
\end{align*}
where
\begin{align*}
H=\frac{\phi}{2c_0}\Big\lbrace\big[-(r^2-s^2)+(r^2-s^2)^2(\zeta +\tau\omega^2)\big]L_1+\big[-(n+1)+3(r^2-s^2)(\zeta +\tau\omega^2]\big)L_2\Big\rbrace.
\end{align*}
On the other hand
\begin{align}\label{K}
\zeta +\tau\omega^2 =\frac{\phi_{ss}}{[\phi -s\phi_s +(r^2-s^2)\phi_{ss}]}.
\end{align}
Substituting \eqref{K} in $H$ implies the result.
\end{proof}

\bigskip

Now, we are ready to prove Theorem \ref{THM2}.

\bigskip

\noindent
{\bf Proof of Theorem \ref{THM2}:} The following holds
\begin{align}
J_{i|j}=\frac{\partial J_i}{\partial x^j}-\frac{\partial J_i}{\partial y^m}\frac{\partial G^m}{\partial y^i}-J_m\frac{\partial ^2G^m}{\partial y^i\partial y^j}.
\end{align}
Thus the mean stretch curvature tensor can be written as follows
\begin{align}\label{0}
\bar{\Sigma} _{ij}=2\bigg(\frac{\partial J_i}{\partial x^j}-\frac{\partial J_j}{\partial x^i}\bigg)-2\bigg(\frac{\partial J_i}{\partial y^m}\frac{\partial G^m}{\partial y^j}-\frac{\partial J_j}{\partial y^m}\frac{\partial G^m}{\partial y^i}\bigg).
\end{align}
Differentiating \eqref{eq12} and \eqref{J} with respect to $x^j$ and $y^m$ respectively, one obtains
\begin{eqnarray}
\nonumber \frac{\partial J_i}{\partial x^j}\!\!&=&\!\!\Big(2K_ux^j+K_s\frac{y^j}{r}+K_va^j\Big)\Big(x^i-s\frac{y^i}{r}\Big)+\Big(2H_ux^j+H_s\frac{y^j}{r}+H_va^j\Big)\Big(a^i-t\frac{y^i}{r}\Big)\\
\!\!&&\!\! +K\bigg(\delta^i_j -\frac{y^i}{r}\frac{y^j}{r}\bigg),\\
\nonumber \frac{\partial J_i}{\partial y^m}\!\!&=&\!\!\-\frac{1}{r}\Big(\delta ^i_m-\frac{y^i}{r}\frac{y^m}{r}\Big)(sK+tH)+\frac{1}{r}\Big(x^m-s\frac{y^m}{r}\Big)\Big(K_s(x^i-s\frac{y^i}{r})+H_s(a^i-t\frac{y^i}{r})-K\frac{y^i}{r}\Big)\\
\!\!&&\!\! +\frac{1}{r}\Big(a^m-t\frac{y^m}{r}\Big)\Big(K_t(x^i-s\frac{y^i}{r})+H_t(a^i-t\frac{y^i}{r})-H\frac{y^i}{r}\Big),\\
\nonumber \frac{\partial G^m}{\partial y^j}\!\!&=&\!\!\ rP\delta ^m_j+\Big(P_s\frac{y^m}{r}+Q_sx^m+R_sa^m\Big)\Big(rx^j-sy^j\Big)+\Big(P_t\frac{y^m}{r}+Q_tx^m+R_ta^m\Big)\Big(ra^j-ty^j\Big)\\
 \!\!&&\!\! +\Big(Py^m+2rQx^m+2rRa^m\Big)\frac{y^j}{r}\label{1}.
\end{eqnarray}
It follows from \eqref{0} - \eqref{1} that
\begin{align*}
\bar{\Sigma}_{ij}=\frac{2}{r}\Big[rT(x^ia^j-x^ja^i)+(sT+Z)(a^iy^j-a^jy^i)+(W-tT)(x^iy^j-x^jy^i)\Big],
\end{align*}
where
\begin{eqnarray}
\nonumber T\!\!&:=&\!\! K_v-2H_u+(H_s-K_t)P+(sK+tH)(Q_t-R_s)+(Q_sH_s-Q_tK_s)(u-s^2)\\
\nonumber\!\!&&\!\! +(R_sH_s-R_tK_s+Q_sH_t-Q_tK_t)(v-st)+(R_sH_t-R_tK_t)(a^2-t^2),\\
\nonumber Z\!\!&:=&\!\! 2sH_u+tH_v+H_s-HP+2(sK+tH)R-(2QH_s+Q_tK)(u-s^2)\\
\nonumber\!\!&&\!\! -(2RH_s+R_tK+2QH_t+Q_tH)(v-st)-(2RH_t+R_tH)(a^2-t^2),\\
\nonumber W\!\!&:=&\!\! 2sK_u+tK_v+K_s-KP+2(sK+tH)Q-(2QK_s+Q_sK)(u-s^2)\\
\nonumber\!\!&&\!\! -(2RK_s+R_sK+2Qk_t+Q_sH)(v-st)-(2RK_t+R_sH)(a^2-t^2).
\end{eqnarray}
This completes the proof.
\qed

%--------------------------------------------------------------------------------------------------------------------------------------------------------------------------


\begin{thebibliography}{00}
\bibitem{B} L. Berwald, {\it \"{U}ber die $n$-dimensionalen Geometrien konstanter Kr\"{u}mmung, in denen
die Geraden die k\"{u}rzesten sind}, Math. Z. {\bf 30}(1929), 449-469.
%----------------------------------------------------------------------------------------------------------------------------------------------------------
\bibitem{CL} G. Chen and L. Liu, {\it On weakly stretch Randers metrics}, J. Math. Analysis. Appl. {\bf 514}(2022), 126257.
%----------------------------------------------------------------------------------------------------------------------------------------------------------
\bibitem{LL}  W. Liu and B. Li, {\it Projectively flat Finsler metrics defined by the Euclidean metric and related 1-forms}, Differ. Geom. Appl.
{\bf 46}(2016), 14-24.
%----------------------------------------------------------------------------------------------------------------------------------------------------------------------
\bibitem{MM0} M. Matsumoto, {\it Theory of Finsler spaces with $(\alpha,\beta)$-metric}, Rep. Math. Phys.
{\bf 31}(1992) 43-84.
%--------------------------------------------------------------------------------------------------------------------------------------------------
\bibitem{MM} M. Matsumoto, {\it On C-reducible Finsler spaces}, Tensor (N.S.) {\bf 24}(1972) 29-37.
%----------------------------------------------------------------------------------------------------------------------------------------------------------------------
\bibitem{MS} M. Matsumoto, S. H\={o}j\={o}, {\it A conclusive theorem for C-reducible Finsler spaces}, Tensor (N.S.) {\bf 32}(1978) 225-230.
%----------------------------------------------------------------------------------------------------------------------------------------------------------------
\bibitem{MSh} M. Matsumoto, C. Shibata, {\it On semi-$C$-reducibility, T-tensor= 0 and  S4-likeness of Finsler spaces}, J. Math. Kyoto Univ. {\bf 19}(1979), 301-314.
%------------------------------------------------------------------------------------------------------------------------------------------------------------------------
\bibitem{NT} B. Najafi and A. Tayebi, {\it Weakly stretch metrics}, Publ Math Debrecen, {\bf 91}(2017), 441-454.
%------------------------------------------------------------------------------------------------------------------------------------------------------------------------
\bibitem{RG} B. Rezaei and M. Gabrani, {\it A class of Berwaldian Finsler metrics}, Acta Math. Acad. Paedagogicae. Nyegy. {\bf 107}(3)(2017), 259-270.
%---------------------------------------------------------------------------------------------------------------------------------------------------------------
\bibitem{ZSh} Z. Shen,  {\it Lectures on Finsler geometry}, World Scientific. (2001).
%---------------------------------------------------------------------------------------------------------------------------------------------------------------
\bibitem{Sh} Z. Shen, {\it Projectively flat Finsler metrics of constant flag curvature}, Trans. American. Math. Soc. {\bf 355}(2003), 1713-1728.
%-------------------------------------------------------------------------------------------------------------------------------------------------------------
\bibitem{TBS} A. Tayebi, M. Bahadori, H. Sadeghi, {\it On spherically symmetric Finsler metrics with some non-Riemannian curvature properties}, J. Geom. Phys. {\bf 163}(2021), 104125.
%------------------------------------------------------------------------------------------------------------------------------------------------------------------------
\bibitem{TB} A. Tayebi, F. Barati, {\it On L-reducible spherically symmetric Finsler metrics}, Differ. Geom. Appl.  {\bf 90}(2023), 102028.
%------------------------------------------------------------------------------------------------------------------------------------------------
\bibitem{TPN} A. Tayebi, E. Peyghan, B. Najafi, {\it On semi C-reducibility of $(\alpha,\beta)$-metrics}, Int. J. Geom. Methods Mod. Phys.
 {\bf 9}(4)(2012) 1250038.
%--------------------------------------------------------------------------------------------------------------------------------------------------------------------------
\bibitem{TS} A. Tayebi, H. Sadeghi, {\it On generalized Douglas-Weyl $(\alpha,\beta)$-metrics}, Acta Math. Sin. Engl. Ser. {\bf 31}(10)(2015) 1611-1620.
\end{thebibliography}
\end{document}